\documentclass[aip,graphicx,preprint]{revtex4-1}

\usepackage{graphicx}

\usepackage{amsmath}
\usepackage{amsthm}
\usepackage{bm}
\usepackage{amssymb}

\usepackage{color}

\usepackage[normalem]{ulem}

\newtheorem{theorem}{Theorem}[section]
\newtheorem{lemma}[theorem]{Lemma}
\newtheorem{proposition}[theorem]{Proposition}

\newtheorem{corollary}[theorem]{Corollary}

\newcommand{\bt}{\begin{theorem}}
\newcommand{\et}{\end{theorem}}
\newcommand{\bl}{\begin{lemma}}
\newcommand{\el}{\end{lemma}}
\newcommand{\bc}{\begin{corollary}}
\newcommand{\ec}{\end{corollary}}
\newcommand{\bp}{\begin{proposition}}
\newcommand{\ep}{\end{proposition}}
\newcommand{\bpf}{\begin{proof}}
\newcommand{\epf}{\end{proof}}
\newcommand{\be}{\begin{equation}}
\newcommand{\ee}{\end{equation}}
\newcommand{\bal}{\begin{align}}
\newcommand{\eal}{\end{align}}
\newcommand{\beq}{\begin{eqnarray}}
\newcommand{\eeq}{\end{eqnarray}}
\newcommand{\ba}{\begin{array}}
\newcommand{\ea}{\end{array}}
\newcommand{\bma}{\begin{bmatrix}}
\newcommand{\ema}{\end{bmatrix}}
\newcommand{\bi}{\begin{itemize}}
\newcommand{\ei}{\end{itemize}}

\newcommand{\comm}[1]{}

\newcommand{\ccG}{\mathcal{G}}
\newcommand{\ccW}{\mathcal{W}}

\newcommand{\cF}{{\mathcal F}}
\newcommand{\cG}{{\mathcal G}}
\newcommand{\cK}{{\mathcal K}}
\newcommand{\cW}{{\mathcal W}}
\newcommand{\wt}{\widetilde}
\newcommand{\bbR}{{\mathbb R}}
\newcommand{\bbP}{{\mathbb P}}
\newcommand{\bbS}{{\mathbb S}}

\begin{document}
\title{Ground states for exponential random graphs}

\author{Rajinder Mavi}
\email{mavi@math.msu.edu}
\homepage{math.msu.edu/~mavi} 
\address{
Department of Mathematics\\
Michigan State University\\
619 Red Cedar Road\\
C212 Wells Hall\\
East Lansing, MI 48824, USA}

\author{Mei Yin}
\email{mei.yin@du.edu}
\homepage{http://www.cs.du.edu/~meiyin/}
\address{Department of Mathematics\\
University of Denver\\
C.M. Knudson Hall, Room 300\\
2390 S. York St.\\
Denver, CO 80208, USA}

\date{\today}

\begin{abstract}
We propose a perturbative method to estimate the normalization constant in exponential random graph models as the weighting parameters approach infinity. As an application, we give evidence of discontinuity in natural parametrization along the critical directions of the edge-triangle model.

\end{abstract}



 \maketitle

\section{Introduction}
\label{subgraph}
Over the last decades, the availability of network data on typically
very large scales has created the impetus for the development of new theories and methods for
modeling and describing the properties of large networks. The introduction of exponential random graphs has
aided in this pursuit, as they are able to capture a wide variety of
common network tendencies by representing a complex global
structure through a set of tractable local features. From the point of view of extremal combinatorics and statistical mechanics, investigations have been focused on the variational principle of the limiting normalization constant, concentration of the limiting probability distribution, phase transitions, and asymptotic structures. See for example Chatterjee and Varadhan \cite{CV2011}, Chatterjee and Diaconis \cite{CD2013}, Radin and Yin \cite{RY}, Lubetzky and Zhao \cite{LZ1, LZ2}, Radin and Sadun \cite{RS1, RS2}, Radin et al. \cite{RRS}, Kenyon et al. \cite{KRRS1}, Yin \cite{Yin1}, Kenyon and Yin \cite{KY}, Aristoff and Zhu \cite{AZ2}, and Chatterjee and Dembo \cite{CD2}. The main techniques used in these papers are variants of statistical physics, but the elegant theory of graph limits as developed by Lov\'{a}sz and coauthors (V.T. S\'{o}s, B. Szegedy, C. Borgs, J. Chayes, K. Vesztergombi, ...) \cite{BCLSV1, BCLSV2, BCLSV3, Lov, LS}, also plays an important role in the interdisciplinary inquiry. Building on earlier work of Aldous \cite{Aldous1} and Hoover \cite{Hoover}, the graph limit theory connects sequences of graphs to a unified graphon space equipped with a cut metric. Though the theory itself is tailored to dense graphs, parallel theories for sparse graphs are likewise emerging. See Benjamini and Schramm \cite{BS}, Aldous and Steele \cite{AS}, Aldous and Lyons \cite{AL}, and Lyons \cite{Lyons} where the notion of local weak convergence is discussed and the recent works of Borgs et al. \cite{BCCZ1, BCCZ2} that are making progress towards enriching the existing $L^\infty$ theory of dense graph limits by developing a limiting object for sparse graph sequences based on $L^p$ graphons.

In this paper, we study the ``standard'' family of exponential graph models in the asymptotic regime as the exponential parameters approach infinity. As the model name indicates, we are associating exponential weights to graphical ensembles. For each $n$ let $\cG_n$ be the ensemble of simple graphs on $n$ vertices and let $\cG = \cup_n \ccG_n$ be the collection on all simple graphs. The exponential weights are defined in terms of  subgraph densities. For any $H \in \cG$ the homomorphism density of $H $ in a graph $G \in \cG$ is defined as the probability that a random map on the vertex set of $H$ into the vertex set of $G$, $V(H) \to V(G)$ is edge preserving.
We write the homomorphism density as
\begin{equation}
\label{t} t(H, G)=\frac{|\text{hom}(H,G)|}{|V(G)|^{|V(H)|}}.
\end{equation}
Now, for all $n$ we define a probability distribution on $\cG_n$ in terms of the homomorphism density. Let $H_1,..,H_d$ be a given selection of simple graphs, where $H_1 $ is an edge: $H_1 = K_2$.  Let $T: \cG_n \to [0,1]^d$, where the components of $T$ are homomorphism densities $\{T(G_n)\}_i=t(H_i, G_n)$. Given $\beta \in \bbR^{d}$, define the functional on $G_n \in \ccG_n$:
\begin{equation}\label{ham}
 T^\beta(G_n) := \beta \cdot T(G_n) =  \sum_{i = 1}^{d} \beta_i t(H_i, G_n) ,
\end{equation}
and weight $G_n \in \cG_n$ by $e^{n^2 T^\beta(G_n)}$. The normalization constant for the ensemble $\cG_n$ is then given by the partition function,
\begin{equation}\label{ptfn}
  Z^\beta_n = \sum_{G_n \in \ccG_n} \exp \left( n^2 T^{\beta} (G_n)\right).
\end{equation}
The terminology for the partition function is borrowed from thermodynamics. In this context, $|\beta|$ is the inverse temperature.
Renormalizing $T^\beta$ in (\ref{ham}) and fixing $\beta/|\beta|$, one obtains a Hamiltonian $H(G_n)  = - \frac{n^2}{|\beta|} T^{\beta}(G_n)$.   As $n\to \infty$ the major contribution to the partition function  concentrates around the thermal states (labeled by $f^\beta$),  in a manner analogous to standard thermodynamic models (see Section \ref{ierg}). These thermal states are well understood in the large temperature regime $|\beta|\rightarrow 0$. For any selection of subgraphs $H_1,..,H_d$ and $|\beta|$ sufficiently small, the associated thermal state lies in the replica symmetric phase \cite{CD2013}, i.e., $f^\beta \equiv \rho$. On the other hand, the ground states, which are defined as the limit of the thermal states $f^\beta$ as $|\beta | \to \infty$, are not so simple. In some cases the ground state is known to be in the replica symmetric phase \cite{CD2013}, while in other cases the ground state concentrates around a simple graph in $\cG$ \cite{YRF}.

Our motivation for this paper comes from the edge-triangle model, obtained by setting $d = 2$ and $H_2 $ a triangle: $H_2 = K_3$. It was shown in \cite{YRF} that there are countably many critical directions of $\beta$, along which the ground state of the model is chosen from finitely many simple graphs with some unknown distribution, and our goal is to develop a mechanism that determines which of these simple graphs is the proper ground state.

\subsection{Graphon topology}
\label{bkgd}
The thermal states $f^\beta$ belong to the space of graph functions ``graphons" which may be understood as generalizations of graphs. The set of graphs $\ccG$ may be embedded into the space of graphons which consist of symmetric measurable functions from $[0,1]^2$ into $[0,1]$,
\begin{equation}
\ccW = \{f:[0,1]^2 \to [0,1] \text{ and } f \text{ symmetric}\}.
\end{equation}
For any $n $ and a graph $G_n \in \ccG_n$, the  graphon representation is the function
\begin{equation}
\label{graphon}
f^{G_n}(x,y)=\left\{%
\begin{array}{ll}
    1, & \hbox{if $(\lceil nx \rceil, \lceil ny \rceil)$ is an edge in $G_n$,} \\
    0, & \hbox{otherwise,} \\
\end{array}%
\right. \quad (x,y) \in [0,1]^2
\end{equation}
where the interval $[0,1] $ may be intuitively thought of as a `continuum' of graph vertices. The distance between graphons is given in terms of the ``cut distance'', defined for $f,h \in \cW$, as
\begin{equation}
\label{cut}
d_{\Box}(f, h)=\sup_{S, T \subseteq [0,1]}\left|\int_{S\times
T}\left(f(x, y)-h(x, y)\right)dx\,dy\right|.
\end{equation}
However, a nontrivial difficulty arrises from the arbitrary labeling of vertices as they are embedded in $[0,1]$. Thus we introduce the equivalence $f \sim f_\sigma$ where $f_\sigma(x,y)  =   f(\sigma x, \sigma y)$, where $\sigma: [0,1] \to [0,1]$ is any measure preserving bijection. We write the quotient space of  graphons under the equivalence $\sim$ as $\widetilde \cW$, and the equivalance class under $\sim$ of $f\in\cW$ as $\wt f$. Incorporating the equivalence relation $\sim$ yields a distance
\begin{equation}
\label{cut2}
\delta_{\Box}(\tilde{f},
\tilde{h})=\inf_{\sigma_1, \sigma_2}d_{\Box}(f_{\sigma_1},
h_{\sigma_2}),
\end{equation}
 where the infimum ranges over all measure
preserving bijections $\sigma_1$ and $\sigma_2$, making
$(\wt\cW, \delta_{\Box})$ a compact metric space (see Section 9.3 of Lov\'{a}sz \cite{Lov}).
With some abuse of notation we also refer to $\delta_{\Box}$ as
the ``cut distance''.

All graphons arise as the limit of some sequence of graphs. Given a graphon $f \in \cW$ one may construct a graph $G_n = G(n, f)$ by selecting iid points $x_1,..,x_n$ uniformly from $[0,1]$ which represent the vertices of $G_n$, and then connect vertices $i,j$ with probability $f(x_i, x_j)$. In this context the expected subgraph density is given by the graphon homomorphism density
\begin{equation}
\label{tt} t(H, f)=\int_{[0,1]^k}\prod_{\{i,j\}\in E(H)}f(x_i, x_j)dx_1\cdots dx_k,
\end{equation}
which generalizes (\ref{t}) and is continuous in the metric $\delta_{\Box}$.
  Indeed, for any graph $H$, the subgraph homomorphism density of random graphs selected by the above construction converges almost surely to the graphon homomorphism density,
\begin{equation}
\lim_{n\rightarrow \infty} t(H, G(n,f))  = t(H,f).
\end{equation}

\subsection{Exponential random graphs}\label{ierg}
 As discussed above, we will define  measures on $\cG_n$ in terms of subgraph densities. For $G_n \in \cG_n$ define the probability
\begin{equation}
\label{pmf2}
\bbP_n^\beta(G_n)=\exp\left(n^2(T^\beta(G_n)) - \psi_n^\beta\right),
\end{equation}
 where we have introduced the normalization   constant (free energy density),
\begin{equation}
\label{psi2} \psi_n^\beta=\frac{1}{n^2}\log Z_n^\beta.
\end{equation}
By replacing the subgraph homomorphism density (\ref{t}) with graphon homomorphism density (\ref{tt}) in (\ref{ham}), $T^\beta$ extends naturally to $(\wt\cW,\delta_{\Box})$.
Since $T^\beta$ is continuous and bounded on the compact set $\wt \cW$, there is a nonempty compact subset $\cK^\beta$ of $\wt \cW$ so that $T^\beta$ is maximized on $\cK^\beta$. Take $I:[0,1] \to \bbR $ as
\begin{equation}
\label{I}
I(u) =  \frac12 u\log  u + \frac12(1 - u) \log  (1 - u ),
\end{equation}
and then extend the domain of $I$ to $\wt\cW$ by
\begin{equation}
I(\tilde{f}) = \int_{[0,1]^2} I (f(x,y)) dxdy,
\end{equation}
where $f$ is any representative element of $\tilde{f}$.
It follows from Lemma 2.1 in
   Chaterjee and Varadhan \cite{CV2011} that $\psi^\beta(\wt f):= T^\beta(\wt f) - I(\wt f)$
   is well defined on $\wt \cW$ and upper semi-continuous under the cut metric $\delta_\Box$. Let $\cF^\beta$ be the subset of $\wt \cW$ where $\psi^\beta$ is maximized. Then like $\cK^\beta$, $\cF^\beta$ is a nonempty compact subset of $\wt \cW$.

For the purpose of this paper, two theorems from Chatterjee and Diaconis \cite{CD2013} (both based on the large deviation result in \cite{CV2011}) merit some
special attention. Together they connect the occurrence of a phase
transition in the exponential random graph model with the solution of a certain
maximization problem. The first theorem (Theorem 3.1 in \cite{CD2013}) states that for any $\beta$, the limiting normalization constant $\lim_{n\to \infty} \psi_n^\beta$ of the exponential random graph always exists and is equal to $\psi^\beta$. The second theorem (Theorem 3.2 in \cite{CD2013}) states that in the large $n$ limit, the quotient image
$\tilde{f}^{G_n}$ of a random graph $G_n$ drawn from (\ref{pmf2})
must lie close to $\cF^\beta$ with high probability,
\begin{equation}
\delta_{\Box}(\tilde{f}^{G_n}, \cF^\beta)\to 0\hbox{ in
probability as } n\to \infty.
\end{equation}
Since the limiting normalization constant $\psi^\beta$ is the generating function for the limiting
expectations of other random variables on the graph space such as expectations and
correlations of homomorphism densities, a phase transition occurs when $\psi^\beta$ is non-analytic
or when $\cF^\beta$ is not a singleton set. Although it is difficult to evaluate $\psi^\beta$ and determine the maximizing set $\cF^\beta$ for most $\beta$, we will derive an efficient method to approximate $\cF^\beta$ and estimate $\psi^\beta$ for $\beta$ sufficiently far from the origin.

\section{The approximation scheme}
\label{approximation}
Take a finite simple graph $H$ with vertex set $V(H)$ and edge set $E(H)$. Consider a graphon $f \in \cW$. For
each $(r, s) \in E(H)$ and each pair of points $x_r, x_s \in [0, 1]$, define
\begin{equation}
\Delta_{H, r, s}f(x_r, x_s):=\int_{[0, 1]^{\left|V(H)\texttt{\char92}\{r, s\}\right|}}\prod_{\substack{(r', s')\in E(H)\\ (r', s')\neq (r, s)}}f(x_{r'}, x_{s'})\prod_{\substack{v\in V(H)\\ v\neq r, s}}dx_v.
\end{equation}
The above definition appears rather complicated, but in essence may be identified with the homomorphism density (\ref{tt}), if we remove edge $(r, s)$ from $H$ and do not integrate out the associated vertices $x_r$ and $x_s$. For $(x, y) \in [0, 1]$, define
\begin{equation}
\label{Delta}
\Delta_H f(x, y):=\sum_{(r, s)\in E(H)} \Delta_{H, r, s} f(x, y),
\end{equation} 
which corresponds to the total homomorphism density generated after all possible ways of removing one edge from $H$. We give some examples to illustrate this idea. When $H$ is an edge, $\Delta_H f(x, y) \equiv 1$. When $H$ is a triangle, by symmetry, $\Delta_H f(x, y)=3\int_{[0, 1]}f(x, z)f(y, z)dz$.

\subsection{Motivation}
\label{motivation}
We will consider as a motivating example the edge-triangle model, which is a $2$-parameter exponential random graph model obtained from (\ref{pmf2}) by setting  $H_1$ to  an edge and $H_2$ to  a triangle. Take $\beta=ru$ for $u=(u_1, u_2)\in \bbS$. Suppose that $u_1>0$ and $u_2<0$, i.e., the $2$-dimensional parameter vector $u$ is pointing towards the $4$th quadrant. Then $\cK^u$ (which agrees with $\cK^\beta$) consists of graphons $f$ that minimize $(u_1/u_2) e+t$, where $e=t(H_1, f)$ denotes the edge density and $t=t(H_2, f)$ denotes the triangle density of $f$, respectively. This implies that $f$, the maximizers of $T^u$ (and hence of $T^\beta$) must lie on the Razborov curve, which is the lower boundary of the feasible region $R$ of edge-triangle homomorphism densities. As $u_1 e+u_2 t$ is a linear function, $f$ must minimize over the convex hull $P$ of $R$. Since $R$ and $P$ only intersect at the points corresponding to Tur\'{a}n graphons, $f$ must be a Tur\'{a}n graphon. This important fact about the structure of $\cK^\beta$ was further used to derive the maximizing graphons of $\cF^\beta$ in \cite{YRF}. Using the boundedness of $I$, we found that $\cF^\beta$ consists of graphons that can be made arbitrarily close to Tur\'{a}n graphons when the magnitude of $\beta$ is sufficiently large, and exactly which Tur\'{a}n graphon is favored by $\cF^\beta$ depends on the direction of $\beta$. However, as nice as these results are, there is ambiguity concerning the optimal Tur\'{a}n graphon along the critical directions of $\beta$, which correspond to normal lines of the convex hull $P$. The subtlety might be due to the fact that Tur\'{a}n graphons, though close to our optimal graphons in cut distance, are not best approximations to the maximizing set $\cF^\beta$. Generalizing from Tur\'{a}n graphons \cite{HN}, we say that a graphon $f$ is random-free if $f=\chi_A$ for some symmetric measurable subset $A$ of $[0,1]^2$. When $\beta$ is finite, Chatterjee and Diaconis \cite{CD2013} showed that the maximizing graphons in $\cF^\beta$ are almost nowhere random-free -- that is the set $\{x\in[0,1]^2: f(x) \in \{0,1\}\}$ has zero measure. We are thus interested in finding coarse-grained graphons that are not random-free but close to $\cF^\beta$ in cut distance, as they are sufficient to distinguish between candidates for the ground state. Resorting to perturbation analysis in the $|\beta| \to \infty$ regime, we will propose a method that keeps track of only the most significant characteristics of the maximizing graphons and  demonstrate its effectiveness.

\subsection{Assumptions}
\label{assumptions}
Let $R$ denote the range of $T$ in $\bbR^d$ and suppose that the boundary $\partial R$ is piecewise analytic. For $w \in \partial R$, let $F_R(w)$ denote the set of all feasible directions of $R$ at $w$, so that $v \in F_R(w)$ implies that there is $\epsilon > 0$ with $w + \epsilon v \in R$. The (internal) tangent cone $C_R(w)$ of $R$ at $w$ is then given by the closure of $F_R(w)$. See Figure \ref{et} for an illustration of this concept in the edge-triangle model. Let $\beta=r u$ for $u \in \bbS^{d-1}$, where $r$ is sufficiently large. Suppose that $T^u$ (and hence $T^\beta$ as $|\beta| \to \infty$) is maximized at a set of  random-free graphons $f =\chi_A$, i.e., $\cK^u$ (which agrees with $\cK^\beta$) consists of random-free graphons only. Further suppose that $f \in \cK^u$ has the property that
      \be \label{derivative1}
         \Delta_{H_i} f =  a_{i}\chi_{A}+b_{i}\chi_{B},
          \ee
    where $\Delta_{H_i}f$ is defined as in (\ref{Delta}) and $B:=[0, 1]^2 \texttt{\char92} A$. We remark that these nice properties that we assumed are enjoyed beyond the edge-triangle model. Denote by $K_k$ a complete graph on $k$ vertices. Since the vertices of the convex hull $P$ of $R$ for $K_2$ and $K_n$ (any $n$), not just $K_2$ and $K_3$ as in the edge-triangle case, are given by Tur\'{a}n graphons, our argument will run through without much modification in these cases. Utilizing geometry of their respective convex hull \cite{EN}, similar analysis may be extended to more general models. We point out in particular that if $f$ is a Tur\'{a}n graphon, then (\ref{derivative1}) holds for any $H_i$ due to symmetry.

\begin{figure}
\centering
\includegraphics[clip=true, height=3in]{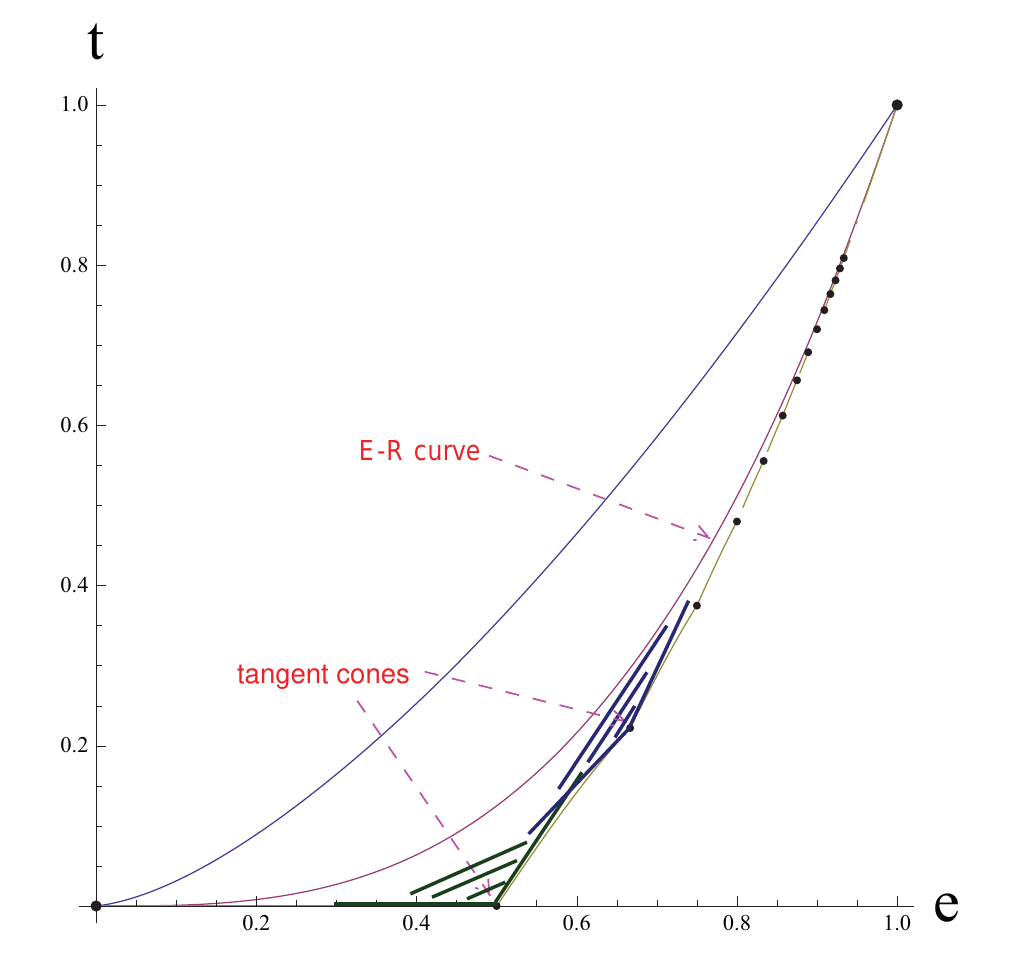}
\caption{Region of attainable edge ($e$) and triangle ($t$) densities for graphons. The upper
boundary is the curve $t=e^{3/2}$ and the lower boundary is a piecewise algebraic curve with
infinitely many concave pieces; see \cite{Razborov}. The tangent cones at complete bipartite graphon (Tur\'{a}n graphon with $2$ classes) and complete tripartite graphon (Tur\'{a}n graphon with $3$ classes) are displayed.} \label{et}
\end{figure}

\subsection{Perturbation of random-free graphons}
\label{perturb}
Consider a maximizing graphon $h \in \cF^\beta$. When the magnitude of $\beta$ is sufficiently large, since $T^\beta$ diverges while $I$ stays bounded, $\cF^\beta$ and $\cK^\beta$ can be made arbitrarily close in the cut metric. From our assumptions in Section \ref{assumptions}, there thus exists a random-free graphon $f=\chi_A$ that is close to $h$ under the cut distance. We will further construct a non-random-free graphon $X_h$ that is close to $h$, simple enough, and yet still retains important information of $h$. Most importantly, we will show that for any finite simple graph $H_i$, $t(H_i, X_h)$ approximates $t(H_i, h)$ at least as well as $t(H_i, f)$ asymptotically. The following proposition is specific for random-free graphons and will be useful for our investigation.

\begin{proposition}
\label{prop}
Let $||\cdot||_1$ denote the $L^1$-norm. For a random-free graphon $f$, $d_\Box(f, h)=O(||f-h||_1)$.
\end{proposition}

\begin{proof}
By (\ref{cut}), it is clear that $d_\Box(f, h) \leq ||f-h||_1$. For the other direction, suppose $f=\chi_A$ for some symmetric measurable subset $A \subset [0, 1]^2$. If $A$ is a rectangle, then $||f-h||_1 \leq 3d_\Box(f, h)$. The general conclusion follows once we recognize that any subset $A$ of $[0, 1]^2$ may be approximated within $\epsilon>0$ by a finite union of disjoint rectangles.
\end{proof}

Let us first expand $t(H_i, h)$ around $f$. Denote a perturbation of $f$ by $g=h-f$. Corresponding to the regime $|\beta| \to \infty$, as explained earlier, $f$ may be chosen so that $|g| \in \cW$. We have
\begin{equation}
\label{exph}
t(H_i, h)=t(H_i, f)+\int_{[0, 1]^2}g(x, y)\Delta_{H_i} f(x, y) dxdy+\text{remainder terms}.
\end{equation}
Under our assumption (\ref{derivative1}), we compute the first variation:
\begin{eqnarray}
\label{middle}
\int_{[0, 1]^2}g(x, y)\Delta_{H_i} f(x, y) dxdy&=&\int_{[0, 1]^2} \left(g(x, y)\chi_A+g(x, y)\chi_B\right)\left(a_i\chi_A+b_i\chi_B\right)dxdy\\\nonumber
&=&a_i\int_A g(x, y)dxdy+b_i\int_B g(x, y)dxdy.
\end{eqnarray}
Define
\begin{equation}
\label{ab}
a=\frac{1}{|A|} \int_A h(x, y) dxdy,
               \hspace{.2in}  b=\frac{1}{|B|} \int_B h(x, y) dxdy.
\end{equation}
Using that $f=\chi_A$ and $g=h-f$, (\ref{middle}) reduces to
\begin{equation}
\int_{[0, 1]^2}g(x, y)\Delta_{H_i} f(x, y) dxdy=a_i(a-1)|A|+b_ib|B|.
\end{equation}
The remainder terms (if they exist) give the homomorphism density after all possible ways of removing at least two edges from $H_i$, and so are bounded above by either $\int_{[0, 1]^3} |g(x, y)||g(x, z)|dxdydz$ or $\int_{[0, 1]^4} |g(x, y)||g(u, v)|dxdydudv$. Estimating the latter is easy.
\begin{equation}
\int_{[0, 1]^4} |g(x, y)||g(u, v)|dxdydudv=\left(\int_{[0, 1]^2}|g(x, y)|dxdy\right)^2=||g||_1^2,
\end{equation}
which is of negligible asymptotic order when compared to the first variation (\ref{middle}). The former corresponds to the $2$-star density of a graphon $|g|$ with edge density $||g||_1$. By \cite{AK}, for small $||g||_1$, the $2$-star density is bounded below by $||g||_1^2$ and above by $||g||_1/2+O(||g||_1^2)$, and the upper bound is achieved when $|g|$ is an anticlique of the form
\begin{equation}
|g(x,y)|=\left\{%
\begin{array}{ll}
    1, & \hbox{if $x>c$ and $y>c$,} \\
    0, & \hbox{otherwise,} \\
\end{array}%
\right.
\end{equation}
where $c=1-\sqrt{1-||g||_1}$. The lower bound for this remainder term is of desirable asymptotic order, but the upper bound is of the same order as the first variation. Recall that the maximizing graphon $h$ for a finite $\beta$ is uniformly bounded away from $0$ and $1$ \cite{CD2013}. The graphon $|g|=|h-f|$ is thus likely quite different from an anticlique. This implies that the $2$-star density of $|h-f|$ is of higher order than $||h-f||_1$ and does not achieve the upper bound. The phenomenon was confirmed for example by simulations for the edge-triangle model \cite{KRRS2}.

Define the averaged perturbation by $X_h = a\chi_A + b \chi_B $, where $a$ and $b$ are given in (\ref{ab}). $X_h$ may be viewed as a flattened out version of $h$. Since
\begin{equation}
\label{agree}
||h-f||_1=\int_A (1-h(x, y))dxdy+\int_B h(x, y)dxdy=(1-a)|A|+b|B|=||X_h-f||_1,
\end{equation}
$X_h$ is close to $f$ (and hence $h$) under the $L^1$ distance, and by Proposition \ref{prop}, also under the cut distance. Denote by $g'=X_h-f$. We perform the same expansion for $t(H_i, X_h)$ around $f$ as in the last paragraph:
\begin{equation}
\label{expXh}
t(H_i, X_h)=t(H_i, f)+\int_{[0, 1]^2}g'(x, y)\Delta_{H_i} f(x, y) dxdy+\text{remainder terms}.
\end{equation}
Following similar reasoning as in (\ref{middle}) and using the definition of $X_h$,
\begin{eqnarray}
\int_{[0, 1]^2}g'(x, y)\Delta_{H_i} f(x, y) dxdy&=&a_i \int_A g'(x, y)dxdy+b_i \int_B g'(x, y) dxdy\\\nonumber&=&a_i(a-1)|A|+b_ib|B|.
\end{eqnarray}
This says that $t(H_i, h)$ and $t(H_i, X_h)$ agree except for the remainder terms. As for $t(H_i, h)$, the remainder terms for $t(H_i, X_h)$ (if they exist) are bounded above by either $\int_{[0, 1]^3} |g'(x, y)||g'(x, z)| dxdydz$ or $\int_{[0, 1]^4} |g'(x, y)||g'(u, v)| dxdydudv$. Since $||g'||_1=||g||_1=(1-a)|A|+b|B|$ by (\ref{agree}), the latter is of asymptotic order $||g||_1^2$; while the former gives
\begin{equation}
\int_{[0, 1]^3} |g'(x, y)||g'(x, z)| dxdydz\leq \left(\max\{1-a, b\}\right)^2,
\end{equation}
and so is also of asymptotic order $||g||_1^2$. We conclude that $t(H_i, X_h)$ gives at least as good an asymptotic approximation for $t(H_i, h)$ as $t(H_i, f)$, and a better one when the error term related to the $2$-star density of $|h-f|$ may be dropped.

\subsection{Maximizing graphons}
\label{maximizing}
Let $\beta=r u$ for $u \in \bbR^d$ and $\|u\|=1$. Under our assumptions, $\cK^u$ (which agrees with $\cK^\beta$) is a set of random-free graphons. Take $h\in \cF^{\beta}$. Then for $\epsilon > 0$ and $r$ large (corresponding to $|\beta| \to \infty$), $d_\Box(f, h)<\epsilon$ for some $f \in \cK^u$, i.e., $f=\chi_A$ for some symmetric measurable subset $A$ of $[0, 1]^2$. By (\ref{exph}) and following analysis,
\begin{equation}
\psi^{\beta}(h)=rT^{u}(f)+r\left(\sum_{i=1}^d a_iu_i\right)(a-1)|A|+r\left(\sum_{i=1}^d b_iu_i\right)b|B|-I(h)+\text{remainder terms}.
\end{equation}
Similarly, by (\ref{expXh}) and following analysis,
\begin{equation}
\psi^{\beta}(X_h)=rT^{u}(f)+r\left(\sum_{i=1}^d a_iu_i\right)(a-1)|A|+r\left(\sum_{i=1}^d b_iu_i\right)b|B|-I(X_h)+\text{remainder terms}.
\end{equation}
In both equations above, the entropy $I$ is bounded in contrast with the energy contribution $T^\beta$. Except that $h$ is close to $f$ (and hence $X_h$) in cut distance, we do not have enough information regarding the structure of $h$. Rather than maximizing $\psi^\beta(h)$ over all possible graphons $h$ directly, we will maximize $\psi^\beta(X_h)$ over $2$-parameter families $0 \leq a, b \leq 1$. From the heuristics in Section \ref{perturb}, this is an effective method to approximate the optimal graphon. More than that, in certain situations (for example the edge-triangle model to be discussed in detail in Section \ref{edgetriangle}), we will see that keeping $2$ parameters is not only effective but also sufficient. Notice that
\begin{equation}
\label{entropyapprx}
I(X_h)=I( a\chi_A + b \chi_B   )=  I(a  )  |A| + I(b) |B|.
\end{equation}
We rewrite $X_h$:
\begin{equation}
\label{flat}
\psi^{\beta}(X_h)=rT^{u}(f)+\left(r\left(\sum_{i=1}^d a_iu_i\right)(a-1)-I(a)\right)|A|+\left(r\left(\sum_{i=1}^d b_iu_i\right)b-I(b)\right)|B|+\text{remainder terms}.
\end{equation}
Maximizing each first variation term, we have
\begin{equation}
\label{optimal}
a = \frac{1}{1 + e^{ - 2r \sum_{i=1}^d a_iu_i} }
               , \hspace{.2in}    b = \frac{1}{1 + e^{ - 2r \sum_{i=1}^d b_iu_i} }.
\end{equation}
Under this choice of $a$ and $b$ and provided we can ignore the remainder terms, $\psi^{\beta}(X_h)$ is strictly bigger than $rT^u(f)$, which is the random-free graphon approximation corresponding to $a=1$ and $b=0$. Let us verify that $\psi^\beta(X_h)$ is indeed strictly bigger than $rT^u(f)$ by rigorously managing the error term in (\ref{flat}). As shown earlier in Section \ref{perturb}, the remainder terms are of higher order:
\begin{multline} \label{ruZ3}
      \psi^{\beta}(X_h)=rT^{u}(f)+\left(r\left(\sum_{i=1}^d a_iu_i\right)(a-1)+O(r(a-1)^2)-I(a)\right)|A|\\+\left(r\left(\sum_{i=1}^d b_iu_i\right)b+O(rb^2)-I(b)\right)|B|.
\end{multline}
For any $\epsilon > 0$, there exists large enough $r$ so that $||X_h-f||_1$ is sufficiently small, making $0<1-a<\epsilon$ and $0<b<\epsilon$. Applying these $\epsilon$ bounds in (\ref{ruZ3}) gives
      \be \label{ruZ5}
      \psi^{\beta}(X_h) \geq  rT^{u}(f)+\left(r\left(\sum_{i=1}^d a_iu_i+\epsilon\right)(a-1)-I(a)\right)|A|+\left(r\left(\sum_{i=1}^d b_iu_i-\epsilon\right)b-I(b)\right)|B|.
     \ee
Maximizing each first variation term as previously, this yields $\psi^{\beta}(X_h)>rT^u(f)$.  We conclude that, as expected, the addition of one more parameter improves the random-free graphon estimation.

\section{The edge-triangle model}
\label{edgetriangle}
Denote by $K_k$ a complete graph on $k$ vertices. The edge-triangle model is a $2$-parameter exponential random graph model obtained by taking $H_1$ an edge ($K_2$) and $H_2$ a triangle ($K_3$) in (\ref{pmf2}). Consider the set $R=\{(t(K_2, f), t(K_3, f)), f\in \cW\}$ of all realizable values of the edge ($e$) and triangle ($t$) homomorphism densities as the graphon $f$ varies over the entire graphon space $\cW$. See Figure \ref{et}. The upper boundary curve of $R$ is given by the equation $t=e^{3/2}$, and can be derived using the Kruskal-Katona theorem (see Section 16.3 of \cite{Lov}). The lower boundary curve is trickier. The trivial lower bound of $t=g_1(e):=0$, corresponding to the horizontal segment, is attainable at any $0 \leq e \leq 1/2$ by graphons describing the possibly asymptotic edge density of subgraphs of complete bipartite graphs (Tur\'{a}n graphon with $2$ classes). For $e \geq 1/2$, the optimal bound was obtained by Razborov \cite{Razborov}, who
established, using the flag algebra calculus, that for $(k-1)/k
\leq e \leq k/(k+1)$ with $k \geq 2$,
\begin{equation}
\label{Ra} t \geq g_k(e):=\frac{(k-1) \left(k-2\sqrt{k(k-e(k+1))}\right)
\left(k+ \sqrt{k(k-e(k+1))} \right)^2}{k^2 (k+1)^2}.
\end{equation}
All the curve segments $g_k(e)$ describing the nontrivial part of the lower boundary of $R$ are strictly concave. For $k=1, 2, \ldots$, we set $v_k:=(e_k, t_k)=(t(K_2, f^{K_{k+1}}), t(K_3, f^{K_{k+1}}))$, where explicitly,
\begin{equation}
f^{K_{k+1}}(x, y)=\left\{%
\begin{array}{ll}
    1, & \hbox{if $\lceil (k+1)x\rceil \neq \lceil (k+1)y \rceil$,} \\
    0, & \hbox{otherwise,} \\
\end{array}%
\right. \quad (x,y) \in [0,1]^2
\end{equation}
is the Tur\'{a}n graphon with $k+1$ classes. Thus
\begin{equation}
e_k=\frac{k}{k+1}, \hspace{.2in} t_k=\frac{k(k-1)}{(k+1)^2}.
\end{equation}
For $k=1, 2, \ldots$, let $L_k$ be the line segment joining vertices $v_k$ and $v_{k+1}$ of neighboring Tur\'{a}n graphons. These infinitely many line segments form the convex hull $P$ of $R$, and the length of $L_k$ decreases monotonically to zero as $k$ gets large.

The normal vectors to $L_k$,
\begin{equation}
\label{ok}
o_k=\left(1, -\frac{(k+1)(k+2)}{k(3k+5)}\right)
\end{equation}
are the critical directions of the edge-triangle model. Let $\beta=ro_k$ and take $r\rightarrow \infty$. While the vectors $o_k$ (\ref{ok}) are not normalized as in our derivation (see Section \ref{approximation}), this can be easily adjusted by adapting $r$. Then Tur\'{a}n graphons with $k+1$ and $k+2$ classes both belong to $\cK^{o_k}$ (which agrees with $\cK^\beta$), and we concluded in \cite{YRF} that a typical graph sampled from the model may behave like either a Tur\'{a}n graphon with $k+1$ classes or a Tur\'{a}n graphon with $k+2$ classes, with no clear preference. Though already quite informative, as explained earlier in Section \ref{motivation}, this result remains somewhat unsatisfactory because it does not indicate whether both such graphons are actually realizable in the limit and in what manner. Notice that underneath our investigation, there is an ordered double asymptotic framework, in the sense that the network size $n$ goes to infinity first followed by the divergence of the parameters $\beta$. In hope of resolving this rather subtle ambiguity within the edge-triangle model, the ``other'' order was also examined in \cite{YRF}, where we first let the magnitude of $\beta$ increase to infinity so as to isolate a simpler sub-model and then study its limiting properties as $n$ grows. Both ordered asymptotics~imply a nearly identical convergence in probability in the cut metric along the noncritical directions. Under the ``original'' as well as the ``reversed'' order, there exist (possibly different) subsequences of the form $\{n_i, \beta_{1,i}, \beta_{2,i}\}$, with $n_i \rightarrow \infty$, $\beta_{1,i}\rightarrow \infty$ and $\beta_{2,i}\rightarrow -\infty$ for $i=1, 2, \ldots$, where the edge-triangle model converges to some Tur\'{a}n graphon specified by the direction of the parameters $(\beta_1, \beta_2)$. Additionally, under the ``reversed'' asymptotics, a detailed categorization of the limiting behavior of the edge-triangle model was obtained: When $\beta$ diverges along the critical direction $o_k$, a typical sampled graph more likely resembles a Tur\'{a}n graphon with $k+2$ classes than with $k+1$ classes. Now that we are equipped with refined perturbation analysis, we would like to sharpen our results under the ``original'' asymptotics and inquire whether the same type of discontinuity in natural parametrization exists.

Let us make a further remark before carrying out the detailed calculations. In the physics literature, people are often interested in cases where the parameter $\beta$ depends on $n$ (some averages need to be satisfied for every $n$). In these models, in place of the normalization constant (free energy density), the relative entropy plays a central role. Analogous (but more complicated) maximization problems and concentration of measure results have been established, which lead to classifications of ensemble equivalence between the microcanonical ensemble and the canonical ensemble. The perturbative methods explored in the current paper are expected to apply in these general parameter situations. In some cases, the perturbation would still be around random-free graphons, and our argument will run through without much adaptation \cite{HMRS}. In some other cases however, the perturbation would be around graphons admitting more intricate structures, and serious future work is needed\cite{GHR,PN}.

\subsection{Perturbation analysis}
Let $\beta=ro_k$. We will compare $\psi^{\beta}(X)$ and $\psi^{\beta}(Y)$, where $X$ is the flattened out graphon close to $f^{K_{k+1}}$ (Tur\'{a}n graphon with $k+1$ classes) and $Y$ is the flattened out graphon close to $f^{K_{k+2}}$ (Tur\'{a}n graphon with $k+2$ classes). Both $X$ and $Y$ are constructed with the optimal perturbation values (\ref{optimal}). We set $u=o_k$ in our calculations below. For $i=1, 2$, we compute $a_i^X$, $b_i^X$, $a_i^Y$, and $b_i^Y$ for $f^{K_{k+1}}$ and $f^{K_{k+2}}$ in (\ref{derivative1}), where $H_1=K_2$ and $H_2=K_3$. As pointed out earlier in Section \ref{approximation}, $a_1^X=b_1^X=a_1^Y=b_1^Y=1$. For notational convenience, from now on we denote $f^{K_{k+1}}$ by $T_{k}$ and $\chi_{[0, 1]^2}-f^{K_{k+1}}$ by $D_{k}$. Then $T_k$ and $D_k$ are indicator functions associated with sets of measure $k/(k+1)$ and $1/(k+1)$, respectively. Reconfirming our assumption in (\ref{derivative1}),
\begin{equation}
\Delta_{K_3} T_k=3\int_{[0, 1]}T_k(x, z)T_k(y, z)dz=\frac{3(k-1)}{k+1}T_k+\frac{3k}{k+1}D_k,
\end{equation}
which gives
\begin{equation}
\label{values}
a_2^X=\frac{3(k-1)}{k+1}, \hspace{.2in} b_2^X=\frac{3k}{k+1},
\end{equation}
\begin{equation*}
a_2^Y=\frac{3k}{k+2}, \hspace{.2in} b_2^Y=\frac{3(k+1)}{k+2}.
\end{equation*}
This yields
\begin{equation}
\label{abs}
\sum_{i=1}^2 a_i^X u_i=\frac{2(k+3)}{k(3k+5)}, \hspace{.2in} \sum_{i=1}^2 b_i^X u_i=-\frac{1}{3k+5},
\end{equation}
\begin{equation*}
\sum_{i=1}^2 a_i^Y u_i=\frac{2}{3k+5}, \hspace{.2in} \sum_{i=1}^2 b_i^Y u_i=-\frac{k+3}{k(3k+5)}.
\end{equation*}

\begin{figure}
\centering
\includegraphics[height=6in, angle=-90]{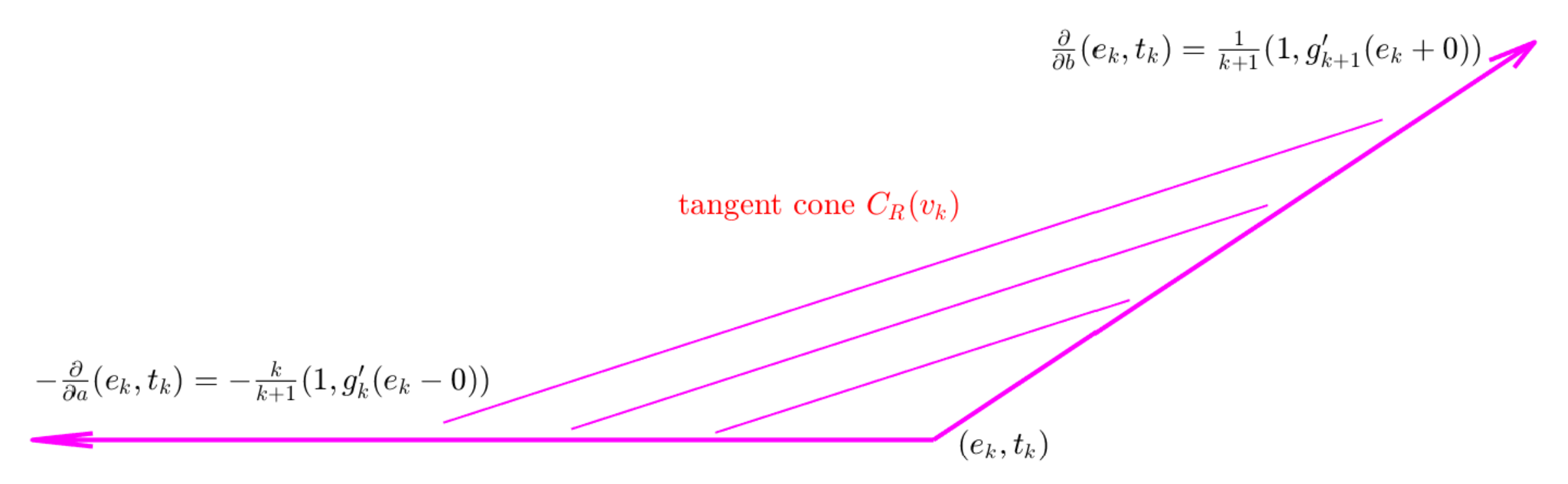}
\caption{The (internal) tangent cone $C_R(v_k)$, consisting of graphons of the type $X_h=aT_{k}+bD_k$. Taking $a=1, b=0$ gives the Tur\'{a}n graphon with $k+1$ classes, whose edge-triangle densities are $v_k=(e_k, t_k)$. Here $g_k'(e_k-0)$ and $g_k'(e_k+0)$ indicate the left and right hand side derivatives of $g_k$ (\ref{Ra}) at $e_k$, respectively.}
\label{cone}
\end{figure}

We check the remainder terms for $\psi^\beta(X)$ after first order perturbation. Similar analysis will work for $\psi^\beta(Y)$ and we skip the details. Consider the (internal) tangent cone $C_R(v_k)$, consisting of graphons of the type $X_h=aT_k+bD_k$. See Figure \ref{cone}. Then
\begin{equation}
\label{edge}
t(K_2, X_h)=\int_{[0, 1]^2}\left(aT_k(x, y)+bD_k(x, y)\right)dxdy=\frac{ak+b}{k+1}.
\end{equation}
For $(x, y) \in [0, 1]^2$, we have
\begin{multline}
\int_{[0, 1]}X_h(x, z)X_h(y, z)dz=\int_{[0, 1]}\left((aT_k(x, z)+bD_k(x, z))(aT_k(y, z)+bD_k(y, z))\right)dz\\
=\left(a^2\frac{k-1}{k+1}+2ab\frac{1}{k+1}\right)T_k(x, y)+\left(a^2\frac{k}{k+1}+b^2\frac{1}{k+1}\right)D_k(x, y),
\end{multline}
which implies that
\begin{equation}
\label{triangle}
t(K_3, X_h)=\int_{[0, 1]^3}X_h(x, y)X_h(x, z)X_h(y, z)dxdydz
\end{equation}
\begin{eqnarray*}
&=&\int_{[0, 1]}\left(aT_k(x, y)+bD_k(x, y)\right)\left(\left(a^2\frac{k-1}{k+1}+2ab\frac{1}{k+1}\right)T_k(x, y)+\left(a^2\frac{k}{k+1}+b^2\frac{1}{k+1}\right)D_k(x, y)\right)dxdy\\\nonumber
&=&a^3\frac{k(k-1)}{(k+1)^2}+3a^2b\frac{k}{(k+1)^2}+b^3\frac{1}{(k+1)^2}. \nonumber
\end{eqnarray*}
This yields
\begin{multline}
\label{exact}
\psi^\beta(X_h)=rt(K_2, X_h)-r\frac{(k+1)(k+2)}{k(3k+5)}t(K_3, X_h)-I(a)\frac{k}{k+1}-I(b)\frac{1}{k+1}\\
=r\frac{ak+b}{k+1}-r\frac{(k+1)(k+2)}{k(3k+5)}\left(a^3\frac{k(k-1)}{(k+1)^2}+3a^2b\frac{k}{(k+1)^2}+b^3\frac{1}{(k+1)^2}\right)-I(a)\frac{k}{k+1}-I(b)\frac{1}{k+1}.
\end{multline}
Using (\ref{optimal}) and (\ref{abs}) for $a$ and $b$, (\ref{exact}) gives $\psi^\beta(X)$. After first order perturbation as in (\ref{flat}), the remainder terms are bounded by
\begin{equation}
\text{constant}\cdot r\cdot (\max\{1-a, b\})^2 \approx \text{constant} \cdot r \cdot \max\big \{\exp\left(-\frac{8r(k+3)}{k(3k+5)}\right), \exp\left(-\frac{4r}{3k+5}\right)\big \}.
\end{equation}
The following lemma is useful for our asymptotic derivation.

\begin{lemma}
\label{approx}
Let $I$ be defined as in (\ref{I}). As $x \rightarrow \infty$,
\begin{equation}
\label{simple1}
-\frac{xe^{-x}}{2(1+e^{-x})}-I(\frac{1}{1+e^{-x}}) \approx \frac{e^{-x}}{2},
\end{equation}
\end{lemma}
\noindent and as $x \rightarrow -\infty$,
\begin{equation}
\label{simple2}
\frac{x}{2(1+e^{-x})}-I(\frac{1}{1+e^{-x}}) \approx \frac{e^{x}}{2}.
\end{equation}
\begin{proof}
We recognize that after simplification, the left hand side of (\ref{simple1}) becomes $\log (1+e^{-x})/2$, and the left hand side of (\ref{simple2}) becomes $x/2+\log (1+e^{-x})/2$. The rest is immediate.
\end{proof}

Applying Lemma \ref{approx} to the first order perturbation terms, we see that conforming to our heuristic analysis in Sections \ref{perturb} and \ref{maximizing}, the remainder terms are indeed of negligible order. For $r$ large enough, we have
\begin{equation}
\label{first}
\psi^{\beta}(X) \approx rT^{o_k}(f)+\frac{k}{2(k+1)}\exp\left(-\frac{4r(k+3)}{k(3k+5)}\right)+\frac{1}{2(k+1)}\exp\left(-\frac{2r}{3k+5}\right),
\end{equation}
\begin{equation*}
\psi^{\beta}(Y) \approx rT^{o_k}(f)+\frac{k+1}{2(k+2)}\exp\left(-\frac{4r}{3k+5}\right)+\frac{1}{2(k+2)}\exp\left(-\frac{2r(k+3)}{k(3k+5)}\right).
\end{equation*}
Since $\sum_{i=1}^2 b_i^X u_i=-1/(3k+5)$ has the smallest absolute value among all equations in (\ref{abs}), $\psi^\beta(X)>\psi^{\beta}(Y)$ both in terms of first order perturbation (\ref{first}) and the exact value.

Denote by $\psi^\beta_\text{opt}(X_h)$ and $\psi^\beta_\text{opt}(Y_h)$ the exact optimal value of $\psi^\beta$ within the (internal) tangent cone $C_R(v_k)$ and $C_R(v_{k+1})$, respectively. Let us compare $\psi^\beta_\text{opt}(X_h)$ with $\psi^\beta_\text{opt}(Y_h)$. From our heuristic argument, for any $\epsilon>0$ and $r$ large enough, (\ref{ruZ5}) is satisfied:
\begin{equation}
\psi^{\beta}(X_h) \geq  rT^{o_k}(f)+\left(r\left(\sum_{i=1}^2 a_i^X u_i+\epsilon\right)(a-1)-I(a)\right)\frac{k}{k+1}+\left(r\left(\sum_{i=1}^2 b_i^X u_i-\epsilon\right)b-I(b)\right)\frac{1}{k+1}.
\end{equation}
Following similar reasoning, we also find a bound in the other direction:
\begin{equation}
\psi^{\beta}(X_h) \leq  rT^{o_k}(f)+\left(r\left(\sum_{i=1}^2 a_i^X u_i-\epsilon\right)(a-1)-I(a)\right)\frac{k}{k+1}+\left(r\left(\sum_{i=1}^2 b_i^X u_i+\epsilon\right)b-I(b)\right)\frac{1}{k+1}.
\end{equation}
Maximizing each first variation term and applying Lemma \ref{approx} as previously, this says that $\psi^\beta(X_h)$ is asymptotically bounded below by
\begin{equation}
rT^{o_k}(f)+\frac{k}{2(k+1)}\exp\left(-2r\left(\frac{2(k+3)}{k(3k+5)}+\epsilon\right)\right)+\frac{1}{2(k+1)}\exp\left(-2r\left(\frac{1}{3k+5}+\epsilon\right)\right),
\end{equation}
and above by
\begin{equation}
rT^{o_k}(f)+\frac{k}{2(k+1)}\exp\left(-2r\left(\frac{2(k+3)}{k(3k+5)}-\epsilon\right)\right)+\frac{1}{2(k+1)}\exp\left(-2r\left(\frac{1}{3k+5}-\epsilon\right)\right),
\end{equation}
Similar analysis works for $\psi^\beta(Y_h)$ and we skip the details. Since $\epsilon>0$ can be taken arbitrarily small and $\sum_{i=1}^2 b_i^X u_i=-1/(3k+5)$ has the smallest absolute value among all equations in (\ref{abs}), $\psi^\beta_{\text{opt}}(X_h)>\psi^{\beta}_{\text{opt}}(Y_h)$. As discussed earlier in Section \ref{subgraph}, by Theorems 3.1 and 3.2 in Chatterjee and Diaconis \cite{CD2013}, we conclude that when $\beta$ diverges along the critical direction $o_k$, a typical sampled graph more likely resembles a Tur\'{a}n graphon with $k+1$ classes than with $k+2$ classes. See Table \ref{table1} and Figure \ref{figure1}.

\begin{theorem}
\label{main}
Consider the edge-triangle exponential random graph model, obtained by setting in (\ref{pmf2}) $H_1$ an edge and $H_2$ a triangle. For $k\geq 1$, let $\beta=ro_k$, where $o_k$ is the critical direction (\ref{ok}) and $r$ is sufficiently large. Then in the large $n$ limit, a typical graph drawn from the model behaves like a Tur\'{a}n graphon with $k+1$ classes,
\begin{equation}
\lim_{r\rightarrow \infty} \sup_{\tilde{f} \in \cF^\beta} \delta_{\Box}(\tilde{f}, \tilde{f}^{K_{k+1}})=0.
\end{equation}
\end{theorem}

\begin{table}
\begin{center}
\begin{tabular}{c||cccccc}
 & $a$ & $b$ & $\psi^{\beta}$ & $a_{\text{opt}}$ & $b_{\text{opt}}$ & $\psi^\beta_{\text{opt}}$ \\
\hline \hline \\
bipartite & 0.999999998 & $0.0759$ & 5.0197 & \text{almost} 1 & 0.069 & 5.019 \\
tripartite & $0.9933$ & $0.0000454$ & 5.0022 & 0.9943 & 0.000064 & 5.0021 \\ \\
\end{tabular}
\end{center}
\caption{Asymptotic comparison for perturbation around neighboring Tur\'{a}n graphons, where $\beta=(10, -7.5)$ diverges along the critical direction $o_1=(1, -0.75)$. $a$, $b$ and $\psi^{\beta}$ are calculated from (\ref{optimal}) and (\ref{first}) with first order perturbation; $a_{\text{opt}}$, $b_{\text{opt}}$ and $\psi^\beta_{\text{opt}}$ are based on numerical optimization for (\ref{exact}). The bipartite feature is associated with a bigger limiting normalization constant and is favored, matching the asymptotic predictions of Theorem \ref{main}.} \label{table1}
\end{table}

\subsection{Geometric interpretation}

We proceed further and examine the effect of infinitesimal perturbation on the associated edge and triangle densities. From (\ref{edge}) and (\ref{triangle}),
\begin{equation}
\label{tangent}
\left.\frac{\partial}{\partial b} \left(t(K_2, X_h), t(K_3, X_h)\right)\right|_{a=1, b=0} = \left( \frac{1}{k+1},  \frac{3k}{(k+1)^2}\right),
\end{equation}
\begin{equation*}
             \left.\frac{\partial}{\partial a} \left( t(K_2, X_h), t(K_3, X_h)\right)\right|_{a=1, b=0} = \left( \frac{k}{k+1},  \frac{3k(k-1)}{(k+1)^2}\right).
\end{equation*}
Let us present another perspective on this calculation incorporating assumption (\ref{derivative1}), which may be employed to derive infinitesimal variations for more complicated homomorphism densities. The idea appeared in our heuristic analysis before and we make it explicit here. For any $H_i$ so that (\ref{derivative1}) is satisfied,
\begin{multline}
\left.\frac{\partial}{\partial b} t(H_i, X_h)\right|_{a=1, b=0}=\int_{[0, 1]^2} D_k(x, y)\Delta_{H_i}T_k(x, y)dxdy\\
=\int_{[0, 1]^2} D_k(x, y)\left(a_iT_k(x, y)+b_iD_k(x, y)\right)dxdy=b_i\frac{1}{k+1}.
\end{multline}
\begin{multline}
\left.\frac{\partial}{\partial a} t(H_i, X_h)\right|_{a=1, b=0}=\int_{[0, 1]^2} T_k(x, y)\Delta_{H_i}T_k(x, y)dxdy\\
=\int_{[0, 1]^2} T_k(x, y)\left(a_iT_k(x, y)+b_iD_k(x, y)\right)dxdy=a_i\frac{k}{k+1}.
\end{multline}
Using $a_1=b_1=1$, $a_2=3(k-1)/(k+1)$ and $b_2=3k/(k+1)$ (\ref{values}), we recover the partial derivatives calculated above. In particular, we recognize that $-(a_1, a_2)$ points along the left tangent line and $(b_1, b_2)$ points along the right tangent line, while the critical direction $u=o_k$ is the normal vector to the line segment $L_k$ that connects neighboring vertices $v_k$ and $v_{k+1}$. This offers a geometric justification of the ($\pm$) signs in (\ref{abs}).

\begin{figure}
\centering
\includegraphics[width=4in]{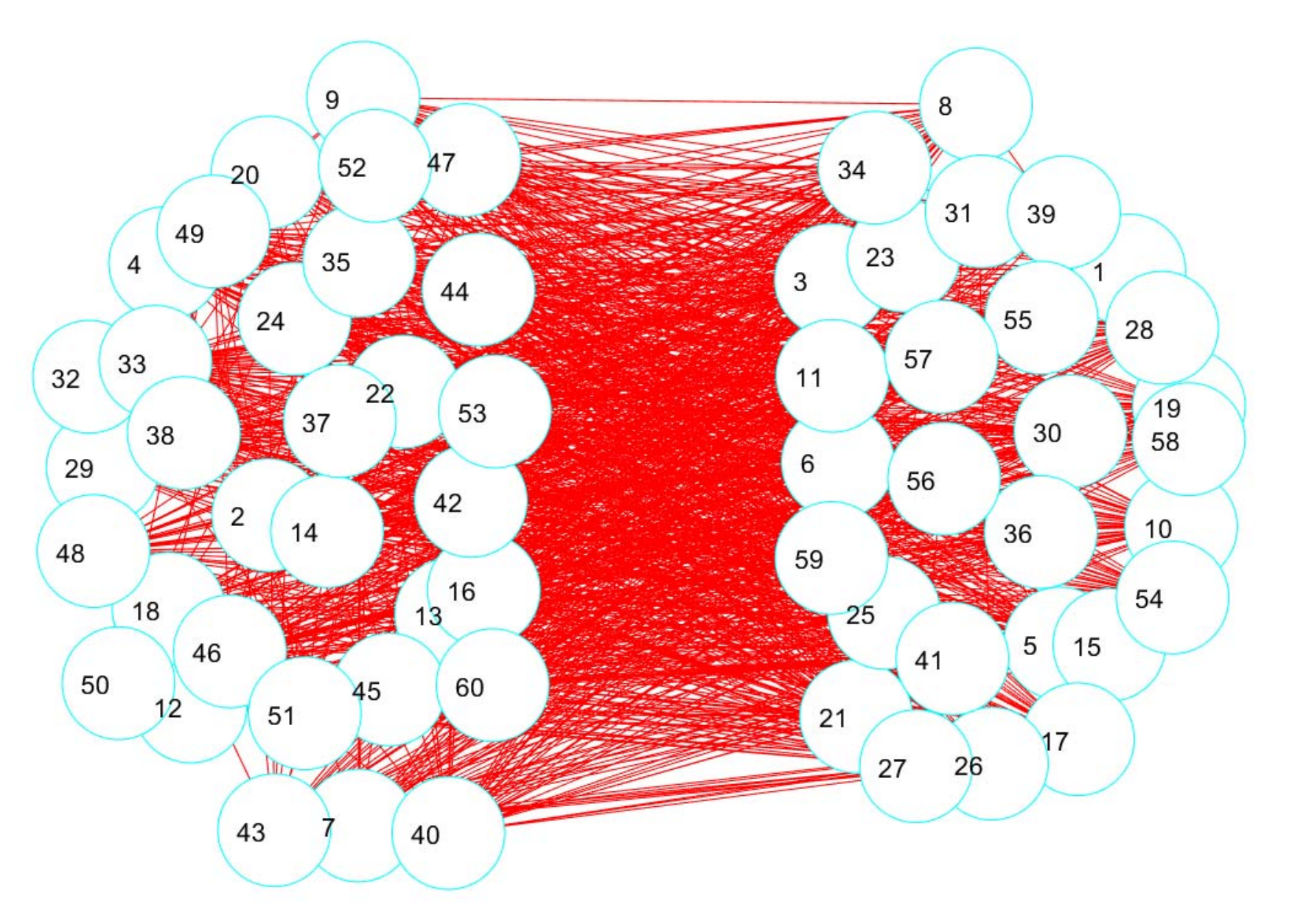}
\caption{A simulated realization of the exponential random graph model on $60$ nodes with edges and triangles as sufficient statistics, where $\beta=(10, -7.5)$ diverges along the critical direction $o_1=(1, -0.75)$. $n$ is of much bigger magnitude than $\beta$ in reflection of the double asymptotic order. The simulated graph displays bipartite feature with edge density $0.557$, matching the asymptotic predictions of Theorem \ref{main}.}
\label{figure1}
\end{figure}

Recall that the lower boundary of attainable edge-triangle densities is a piecewise algebraic curve with infinitely many concave
pieces $g_k$ (\ref{Ra}), and the connection point of $g_k$ and $g_{k+1}$ is $v_k=(e_k, t_k)$. See Figure \ref{et}. We compute
\begin{equation}
g_k'(e)=\frac{3(k-1)}{k(k+1)}\left(k+\sqrt{k(k-e(k+1))}\right),
\end{equation}
which implies that the left and right hand side derivatives at $e_k=k/(k+1)$ are respectively given by
\begin{equation}
g_k'(e_k-0)=\frac{3(k-1)}{k+1}, \hspace{.2in} g_{k+1}'(e_k+0)=\frac{3k}{k+1}.
\end{equation}
The partial derivative vectors in (\ref{tangent}) thus delineate the boundary of the (internal) tangent cone $C_R(v_k)$. In other words, $C_R(v_k)$ spans all possible infinitesimal variations at the Tur\'{a}n graphon with $k+1$ classes. See Figure \ref{cone}. Since the optimizing graphon associated with large enough $\beta$ must lie within the tangent cone of some Tur\'{a}n graphon, it may be represented by a linear combination of Erd\H{o}s-R\'{e}nyi and Tur\'{a}n graphons. Even though the graphon representation may not be unique, optimizing over all possible combinations provides insight into the structure of the maximizing set. Keeping track of the Erd\H{o}s-R\'{e}nyi and Tur\'{a}n characteristics in the edge-triangle model is thus not only an effective but also sufficient method to estimate the normalization constant, and gives evidence of discontinuity of the natural parametrization along the critical directions $o_k$ in the limit as $n$ and then $r$ tend to infinity. This demonstrates the occurrence of discontinuous phase transitions in the edge-triangle model.

\section*{Acknowledgements}
The authors are very grateful to the anonymous referee for the invaluable
suggestions that greatly improved the quality of this paper. Mei Yin thanks Sukhada Fadnavis for helpful conversations. Rajinder Mavi was supported by a postdoctoral fellowship from the Michigan State University Institute for Theoretical and Mathematical Physics. Mei Yin's research was partially supported by NSF grant DMS-1308333.


\end{document}